\documentclass[a4paper,12pt,reqno]{amsart}
\usepackage{amsmath,amsthm,amsfonts,amssymb}
\usepackage{indentfirst}
\usepackage[numbers]{natbib}
\usepackage{a4wide}
\usepackage{url}
\usepackage{graphicx}
\usepackage{hyperref}

\theoremstyle{plain}
\newtheorem{teo}{Theorem}[section]
\newtheorem{cor}[teo]{Corollary}

\theoremstyle{definition}
\newtheorem{defn}[teo]{Definition}

\theoremstyle{remark}
\newtheorem{obs}[teo]{Remark}

\numberwithin{equation}{section}

\newcommand{\bbR}{\mathbb{R}}

\newcommand{\eps}{\varepsilon}
\newcommand{\V}{\text{Var}}
\newcommand{\Cov}{\text{Cov}}

\newcommand{\xf}{x_{\infty}}
\newcommand{\tf}{\tau_{\infty}}
\newcommand{\n}[2]{\ensuremath{{#1}^{(#2)}}}
\newcommand{\s}[3]{\ensuremath{{#1}^{(#2)}_{#3}}}

\begin{document}

\title[Rumour process with random stifling]{On the behaviour of a rumour process \\
with random stifling}

\author[Elcio Lebensztayn et al.]{Elcio Lebensztayn}
\address[E. Lebensztayn and F. P. Machado]{Statistics Department, Institute of Mathematics and Statistics, University of S\~ao Paulo, Rua do Mat\~ao 1010, CEP 05508-090, S\~ao Paulo, SP, Brazil.}
\email{elcio@ime.usp.br}
\thanks{Elcio Lebensztayn was supported by CNPq (311909/2009-4), F\'abio P. Machado by CNPq (306927/2007-1), and Pablo M. Rodr\'iguez by FAPESP (2010/06967-2).}
\author[]{F\'abio P. Machado}
\email{fmachado@ime.usp.br}
\author[]{Pablo M. Rodr\'iguez}
\address[P. M. Rodr\'{i}guez]{Statistics Department, Institute of Mathematics, Statistics and Scientific Computing, State University of Campinas, Rua S\'ergio Buarque de Holanda 651, CEP 13083-859, Campinas, SP, Brazil.}
\email{pablor@ime.unicamp.br}

\date{\today}

\begin{abstract}
We propose a realistic generalization of the Maki-Thompson rumour model by assuming that each spreader ceases to propagate the rumour right after being involved in a random number of stifling experiences. 
We consider the process with a general initial configuration and establish the asymptotic behaviour (and its fluctuation) of the ultimate proportion of ignorants as the population size grows to $\infty$. 
Our approach leads to explicit formulas so that the limiting proportion of ignorants and its variance can be computed.
\end{abstract}

\maketitle 

\section{Introduction}
\label{S: Introduction}

In the past decades, there has been great interest in understanding and modelling different processes for information diffusion in a population. Most of the time, the mathematical theory of epidemics is adapted for this purpose, even though there are differences between the process of spreading information and the process of spreading a virus or a disease. 
In the standard versions of the models, the most noticeable differences are between the way spreaders cease to spread an item of information and the way infected individuals are removed from epidemic processes. 
Still, some slightly modified models fit both processes (see, for example, \citet{dunstan}, where the general stochastic epidemic model is considered as a model for the diffusion of rumours).

\citet{kurtz} recently introduced a model using a complete graph in which, as soon as an individual is infected, an anti-virus is given to that individual in such a way that the next time a virus tries to infect it, the virus is ineffective. 
Besides, a virus can survive up to $ L $ individuals empowered with anti-virus. 
Individuals are represented by the vertices of the complete graph, while the virus is represented by a moving agent that replicates every time it hits a healthy individual. The authors prove a Weak Law of Large Numbers and a Central Limit Theorem for the proportion of infected individuals after the process is completed. 

There are two classical models for the spreading of a rumour in a population, which were formulated by \citet{kendall} and \citet{maki}.
In the model proposed by \citet{maki}, a closed homogeneously mixing population experiences a rumour process. Three classes of individuals are considered: ignorants, spreaders and stiflers. The rumour is propagated through the population by directed contact between spreaders and other individuals, which are governed by the following set of rules.
When a spreader interacts with an ignorant, the ignorant becomes a spreader; whenever a spreader contacts a stifler, the spreader turns into a stifler and when a spreader meets another spreader, the initiating spreader becomes a stifler. In the last two cases, it is said that the spreader was involved in a \textit{stifling experience}. 
Observe that the process eventually ends (when no more spreaders are left in the population).

We show how the techniques used by \citet{kurtz} in the context of epidemic models can be useful in studying a general rumour process. In particular, we propose a generalization of the Maki-Thompson model. In our model, each spreader decides to stop propagating the rumour right after being involved in a random number of stifling experiences.

To define the process, consider a closed homogeneously mixing population of size $N + 1$.
Let $R$ be a nonnegative integer valued random variable with distribution given by $P(R=i)=r_i$ for $i=0,1,\dots$, and let $\mu = E[R]>0$ and $\nu^2=\V[R]$. 
Assign independently to each initially ignorant individual a random variable with the same distribution as $R$. 
Once an ignorant hears the rumour, the value of $R$ assigned to him determines the number of stifling experiences the new spreader will have until he stops propagating the rumour. If this random variable equals zero, then the ignorant joins the stiflers immediately after hearing the rumour.

For $i=1,2,\dots$, we say that a spreader is of type~$i$ if this individual has exactly $i$ remaining stifling experiences.
We denote the number of ignorants, spreaders of type~$i$ and stiflers at time $t$ by $\n{X}{N}(t)$, $\n{Y_i}{N}(t)$ and $\n{Z}{N}(t)$, respectively.
Let $\n{Y}{N}(t) = \sum_{i=1}^{\infty} \n{Y_i}{N}(t)$ be the total number of spreaders at time~$t$, so $\n{X}{N}(t)+\n{Y}{N}(t)+\n{Z}{N}(t)=N+1$ for all~$t$. Notice that the infinite-dimensional process 
\begin{equation}
\label{F: Proc}
\{\n{V}{N}(t)\}_{t\geq 0}:=\{(\n{X}{N}(t),\n{Y}{N}_1(t),\n{Y}{N}_2(t),\dots)\}_{t\geq 0}
\end{equation}
is a continuous time Markov chain with increments and corresponding rates given by
{\allowdisplaybreaks
\begin{align*}
&\text{increment} & &\text{rate} & &  \\
&(-1,0,0,\dots) & &r_0 \, X Y\\
&(-1, 0, \dots, \overset{i-1}{0},\overset{i}{1},\overset{i+1}{0},\dots) & &r_i \, X Y & & i=1,2,\dots\\
&(0,\dots,0,\overset{i-1}{1},\overset{i}{-1}, \overset{i+1}{0}, \dots) & &\left(N-X\right) Y_i & & i=2,3,\dots \\[0.1cm]
&(0,-1,0,0,\dots) & &\left(N-X\right) Y_1. & & 
\end{align*}}%
We see that the first case indicates the transition of the process in which a spreader interacts with an ignorant and
the ignorant becomes a stifler immediately (which happens with probability~$r_0$). 
The second case indicates the transition in which a spreader interacts with an ignorant and
the ignorant becomes a spreader of type $i$ (which happens with probability $r_i$).
The third case represents the situation in which a spreader of type~$i$ is involved in a stifling experience but remains a spreader (of type $i-1$), and finally the last transition indicates the event that a spreader of type $1$ is involved in a stifling experience, thus becoming a stifler.

We suppose that the process starts with
\begin{align*}
&\n{X}{N}(0)=(N+1) \, \s{x}{N}{0},  \\
&\n{Y}{N}_i(0)=(N+1) \, \s{y}{N}{i,0} \; \text{ for } i=1,2,\dots \; \text{and }\\
&\n{Z}{N}(0)=(N+1) \, \s{z}{N}{0}.
\end{align*}
That is, $\s{x}{N}{0}, \s{y}{N}{i,0},\s{z}{N}{0} \in [0,1]$ are the initial proportions of ignorants,
spreaders of type $i$ and stiflers of the population, respectively, which are defined in such a way that
\begin{gather*}
\sum_{i=1}^{\infty} \s{y}{N}{i,0} > 0 \quad \text{and} \\
\s{x}{N}{0} + \sum_{i=1}^{\infty} \s{y}{N}{i,0} + \s{z}{N}{0} = 1.
\end{gather*}
In addition, we assume that the following limits exist:
\begin{align*}
x_0 &= \lim_{N \to \infty} \s{x}{N}{0} > 0 \quad \text{and} \\
y_{i,0} &= \lim_{N \to \infty} \s{y}{N}{i,0} \; \text{ for all } i = 1,2, \dots,
\end{align*}
and that
\[ w_0 = \sum_{i=1}^{\infty} i \, y_{i,0} < \infty. \]

As already mentioned, the process eventually ends. Let
$$ \tau^{(N)}= \inf \{t: \n{Y}{N}(t) = 0 \} $$
be the absorption time of the process. Our main purpose is to establish limit theorems for the proportion $N^{-1}\n{X}{N}(\tau^{(N)})$ of ignorants at the end of the process. 
For the classical Maki-Thompson model, this problem was first studied rigorously 
by \citet{sudbury}, who proved, by using martingale arguments, that this proportion converges in probability to $0.203$. 
This result was later generalized by \citet{watson} using the normal asymptotic approximation. 
\citet{picard} derived the exact joint distribution of the final number of people who heard the rumour and the total personal time units during which the rumour was spread. 
In \citet{belen04}, the authors present an analysis of the proportion of the population who never hear the rumour starting from a general initial condition.
See also Chapter~5 of \citet{dg} for an excellent account of rumour models.

The approach used to prove our theorems is the theory of density dependent Markov chains, presented in~\citet{ethier}. 
To the best of our knowledge, this technique is used for the first time in the context of rumour models here and in \citet{EFP}.
In that paper, the authors study a family of rumour processes which includes the classical Daley-Kendall and Maki-Thompson models as particular cases.
The results presented here are of independent interest, as they refer to a generalization of the Maki-Thompson model with random stifling and general initial configuration.

\section{Main results}
\label{S: Main results}

\begin{defn}
\label{D: fun}
Suppose that $\mu <\infty$ and consider the function $f:(0,x_0]\longrightarrow \bbR$ given by
\begin{equation*}
f(x)=w_0+(1+\mu)(x_0-x)+\log \frac{x}{x_0}.
\end{equation*}
We define $\xf =\xf(\mu,x_0,w_0)$ as the unique root of $f$ in the interval $(0,x_0]$ satisfying $f'(x)\geq 0$.  
\end{defn}

Notice that $\xf$ is the unique root of $f$, except in the case where $x_0>(1+\mu)^{-1}$ and $w_0=0$. 
See Figure~\ref{Fig: Root}.

\begin{figure}[ht]
\begin{center}
\includegraphics[scale=0.48]{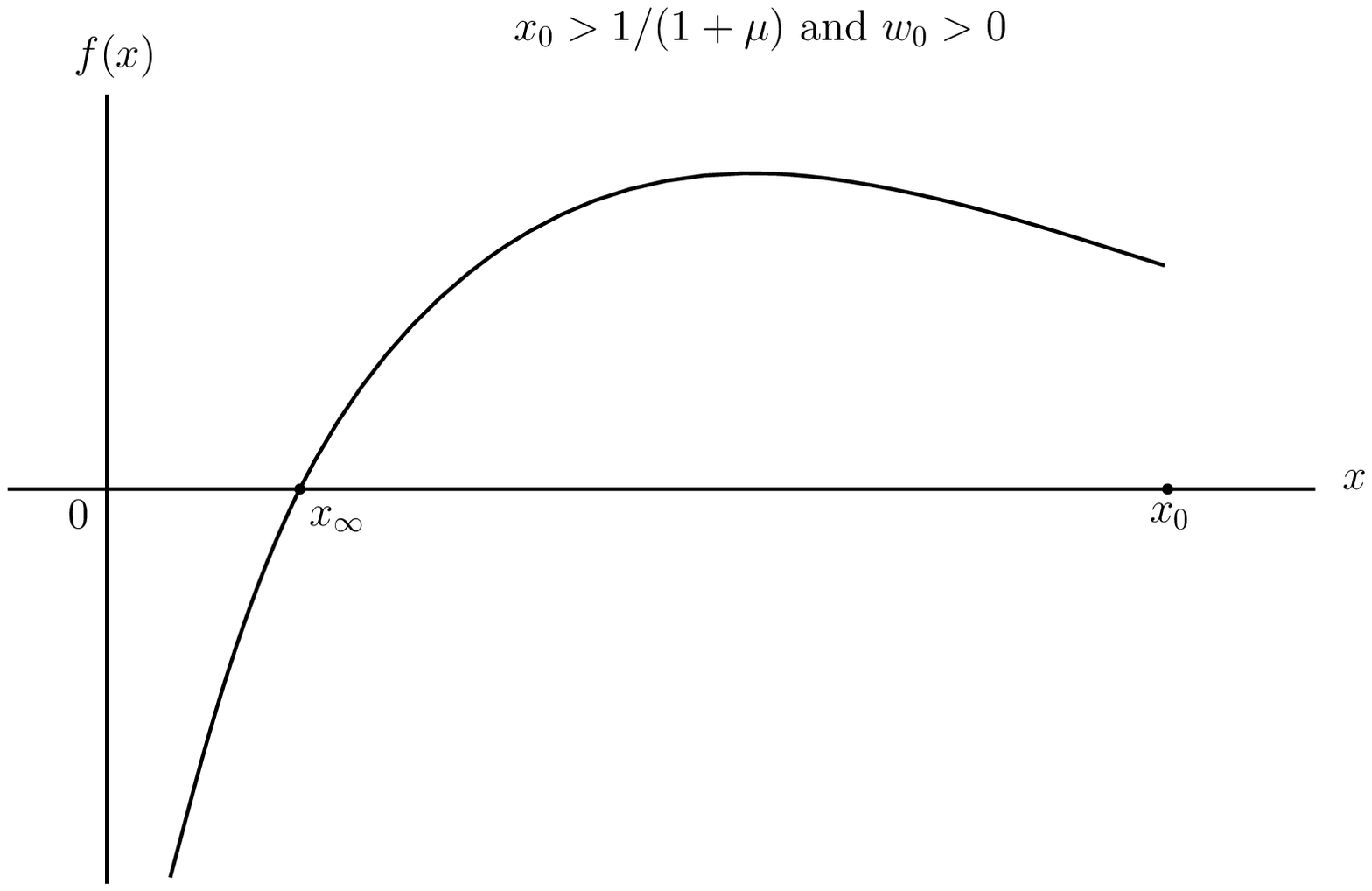}
\quad
\includegraphics[scale=0.48]{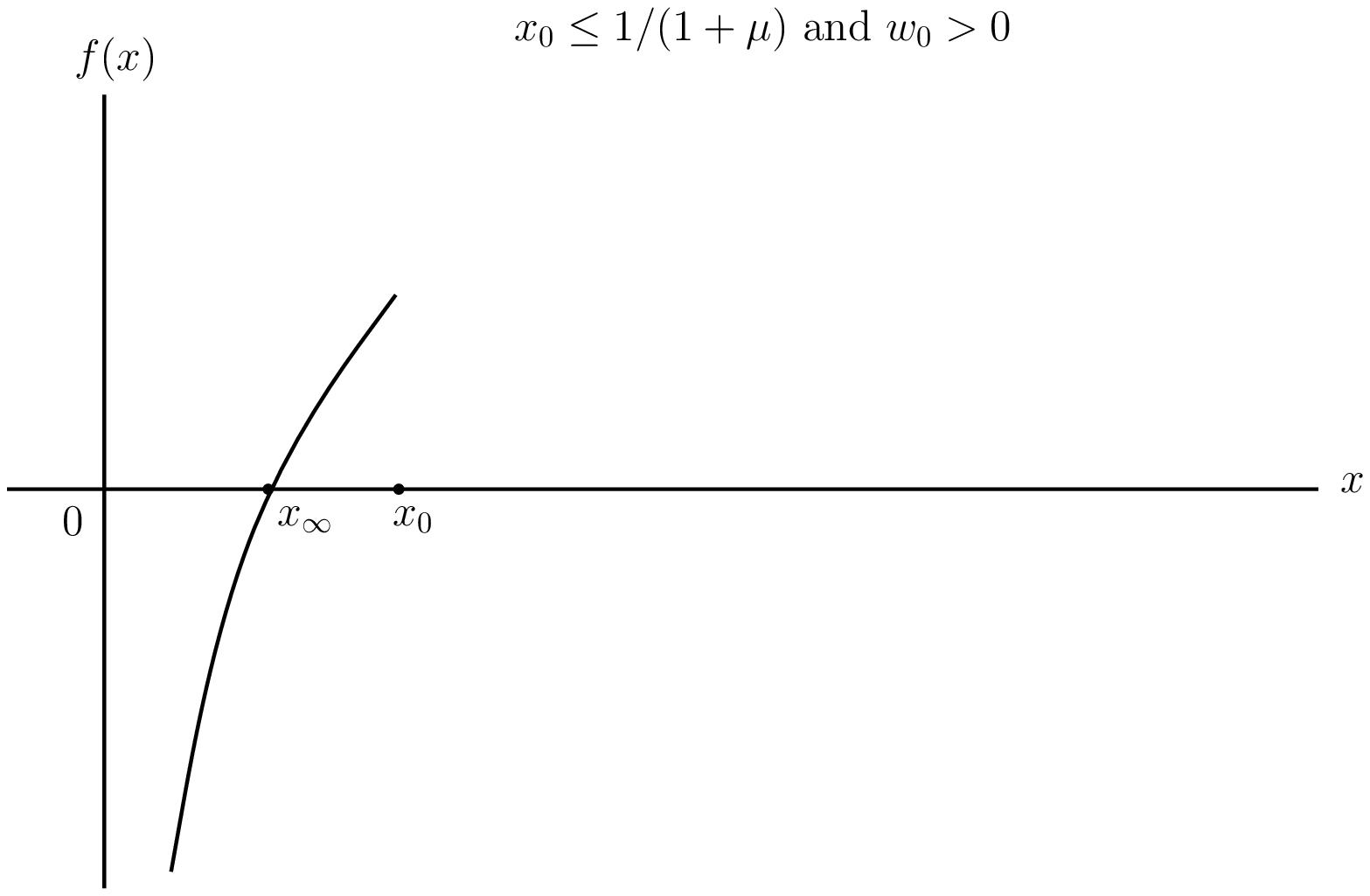}
\\[0.2cm]
\includegraphics[scale=0.48]{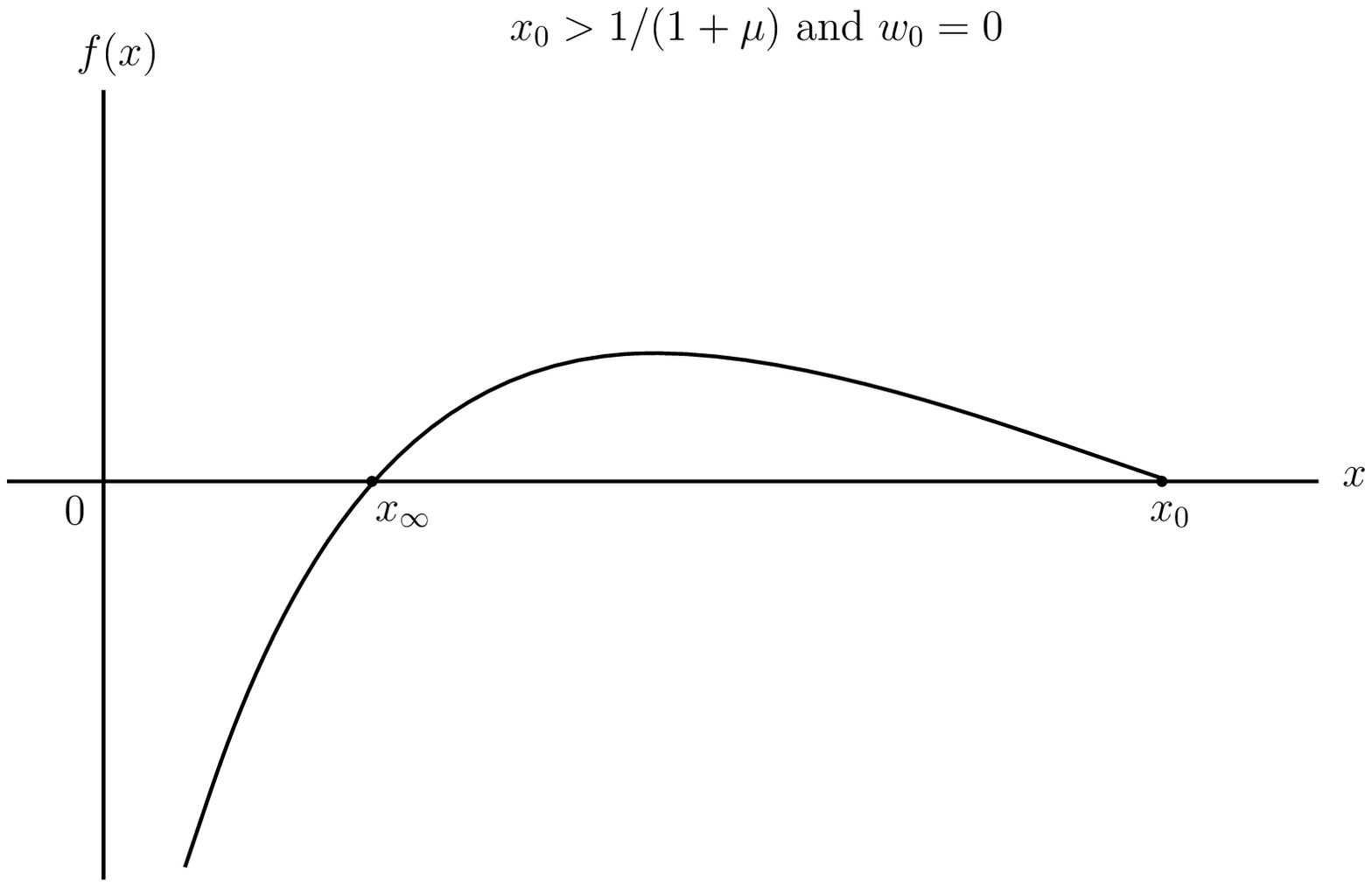}
\quad
\includegraphics[scale=0.48]{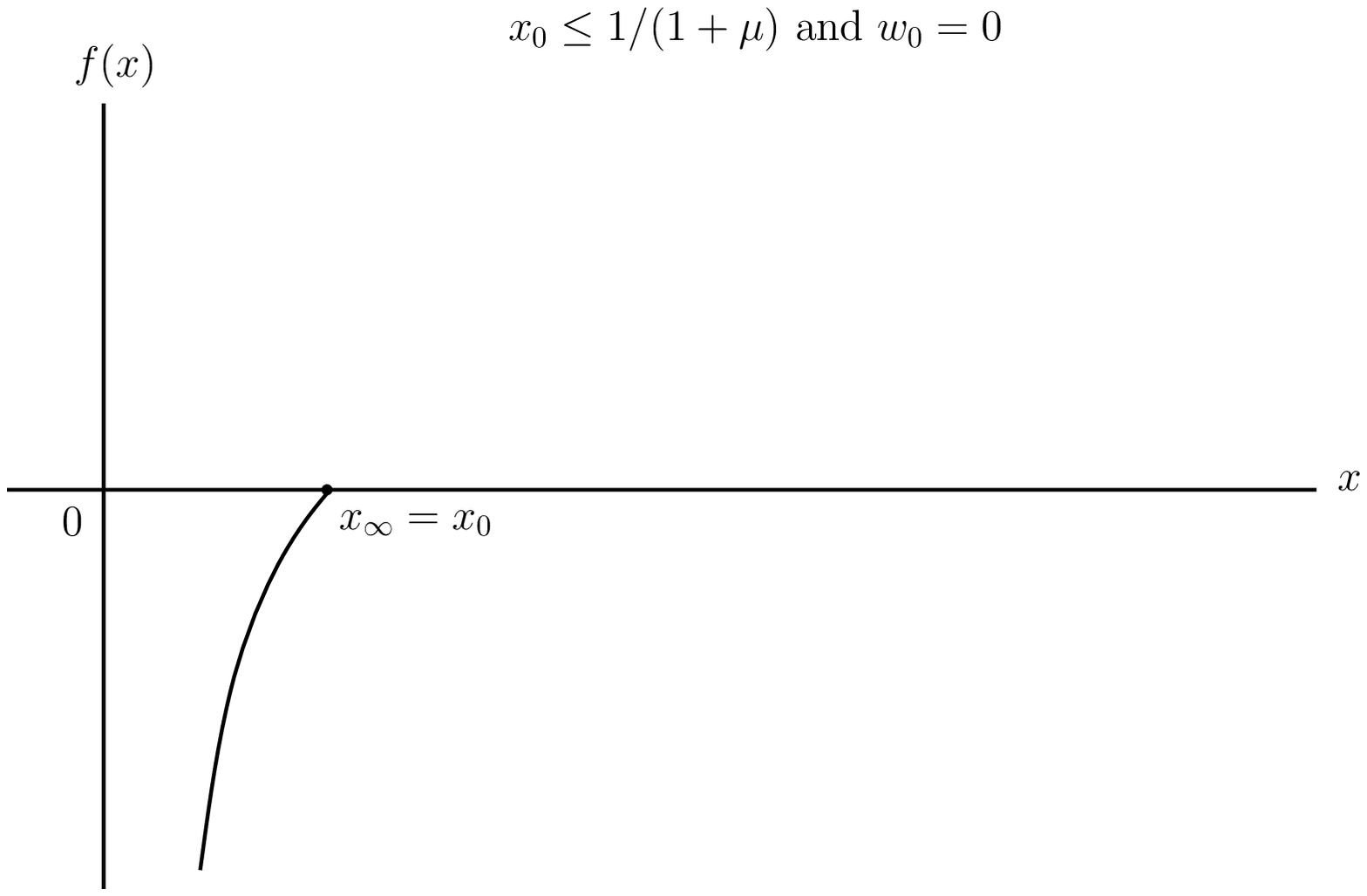}
\end{center}
\caption{Behaviour of $ f $ -- The four possible cases in terms of $x_0$ and $w_0$.}
\label{Fig: Root}
\end{figure}

\begin{obs}
We can express $\xf$ in terms of the Lambert $W$ function, which is the inverse of the function $ x \mapsto x \, e^x $. Indeed, $\xf$ satisfies
$$\xf = x_0 \, e^{-(1+\mu)(x_0-\xf)-w_0}$$
which can be written as
\begin{equation}
\label{F: prelambert}
-x_0(1+\mu) \, e^{-x_0(1+\mu)-w_0}=-\xf(1+\mu) \, e^{-\xf(1+\mu)}.
\end{equation}
Then, if $W_0$ denotes the principal branch of the Lambert $W$ function (that is, the branch that satisfies $W(x)\geq -1$), we obtain from \eqref{F: prelambert} that
\begin{equation}
\label{F: lambert}
\xf(\mu, x_0, w_0) = -(1+\mu)^{-1} \, W_0(-x_0 \, (1+\mu) \, e^{-x_0(1+\mu)-w_0}),
\end{equation}
by noting that $-e^{-1}<-x_0(1+\mu) \, e^{-x_0(1+\mu)-w_0}<0.$ More details about the Lambert function can be found in \citet{lambert}.
\end{obs}

Next, we state the Weak Law of Large Numbers for the proportion of the population who have never heard the rumour.

\begin{teo}
\label{T: LLN}
If $0<\mu<\infty$, then
$$\lim_{N\to \infty}\frac{\n{X}{N}(\tau^{(N)})}{N}=x_\infty \quad \text{in probability. }$$
\end{teo}

As a consequence of this theorem, 

\begin{cor}
\label{C: muinf}
If $\mu=\infty$, then
$$\lim_{N\to \infty}\frac{\n{X}{N}(\tau^{(N)})}{N}=0 \quad \text{in probability.}$$
\end{cor}

\begin{proof}
Let $R_1$ and $R_2$ be nonnegative integer valued random variables such that $R_1 \leq R_2$ stochastically, that is, $P(R_1\geq i)\leq P(R_2\geq i)$ for all $i\geq 1$. Consider the processes $\{{V}_1^{(N)}(t)\}_{t\geq 0}$ and $\{{V}_2^{(N)}(t)\}_{t\geq 0}$ defined as in~\eqref{F: Proc} by using the random variables $R_1$ and~$R_2$, respectively, and with the same initial conditions. Let $\tau_1^{(N)}$ and $\tau_2^{(N)}$ be the respective absorption times. 
By a coupling argument, these processes can be constructed in such a way that
\begin{equation}
\label{F: comp}
{X}_2^{(N)}(\tau_2^{(N)})\leq {X}_1^{(N)}(\tau_1^{(N)}), \quad \text{a.s.}
\end{equation}
Now suppose that $\{{V}^{(N)}(t)\}_{t\geq 0}$ is the process defined in~\eqref{F: Proc} with the random variable $R$ satisfying $\mu=\infty$. Recall that $r_i = P(R=i)$, $i\geq 0$, and for each $k\geq 1$ define the random variable $R_k$ with distribution given by
\[ P(R_k=i)=r_i \, \text{ if } \, i < k \quad \text{and} \quad P(R_k=k)=\sum_{j=k}^{\infty}r_j. \]
By construction, we have that $R_k \leq R$ stochastically for all $k$. Taking this and \eqref{F: comp} into account, we conclude that
\begin{equation}
\label{F: mono}
{X}^{(N)}(\tau^{(N)})\leq {X}_k^{(N)}(\tau_k^{(N)}), \quad \text{a.s.}
\end{equation} 
for all $k$. Then \eqref{F: mono} and Theorem \ref{T: LLN} imply that
\begin{equation*}
0 \leq \limsup_{N \to \infty}\frac{X^{(N)}(\tau^{(N)})}{N}\leq \xf(\mu_k,x_0,w_0) \quad \text{a.s.},
\end{equation*}
where $\mu_k=E[R_k]$. 
Since $x_0>0$ and $\lim_{k \to \infty}\mu_k=\infty$, we have that $x_0>(1+\mu_k)^{-1}$ for large enough $k$ and in this case $\xf(\mu_k,x_0,w_0)$ (given by~\eqref{F: lambert} with $\mu = \mu_k$) goes to $0$ as $k\to \infty$. 
This completes the proof of Corollary \ref{C: muinf}.
\end{proof}

We now present the Central Limit Theorem for the ultimate proportion of ignorants in the population.

\begin{teo}
\label{T: CLT}
Suppose that $\nu^2<\infty$. Assume also that $w_0>0$ or that $w_0=0$ and $x_0>(1+\mu)^{-1}$. Then,
\[ \sqrt{N} \left(\frac{\n{X}{N}(\tau^{(N)})}{N} - x_\infty \right) \Rightarrow N(0, \sigma^2) \, \text{ as } \, N \to \infty, \]
where $ \Rightarrow $ denotes convergence in distribution, and $ N(0, \sigma^2) $ is the Gaussian distribution with mean zero and variance given by
\begin{equation}
\label{F: Variance}
\sigma^2 = \frac{\xf (1 - (x_0^{-1} + w_0 + (x_0 - \xf) (1 + \mu - \nu^2)) \, \xf)}{(1 - (1 + \mu) \, \xf)^2}.
\end{equation}
\end{teo}

\begin{obs}
Observe that our results refer to a general initial condition, similar to that considered in the deterministic analysis presented in~\citet{belen04}.
The process starting with one spreader and $N$~ignorants corresponds to $x_0=1$ and $w_0=0$, in which case the limiting fraction of ignorants and the variance of the asymptotic normal distribution in the CLT reduce respectively to
\begin{gather*}
\xf = \xf(\mu, 1, 0) = -(\mu + 1)^{-1} \, W_0(-(\mu + 1) \, e^{-(\mu + 1)}) \quad \text{and} \\[0.1cm]
\sigma^2 = \frac{\xf (1 - \xf) (1 - (1 + \mu - \nu^2) \, \xf)}{(1 - (1 + \mu) \, \xf)^2}.
\end{gather*}
The behaviour of $\xf$ as a function of $\mu$ is shown in Figure~\ref{Fig: Proportion}.

\begin{figure}[ht]
\centering
\includegraphics[scale=0.9]{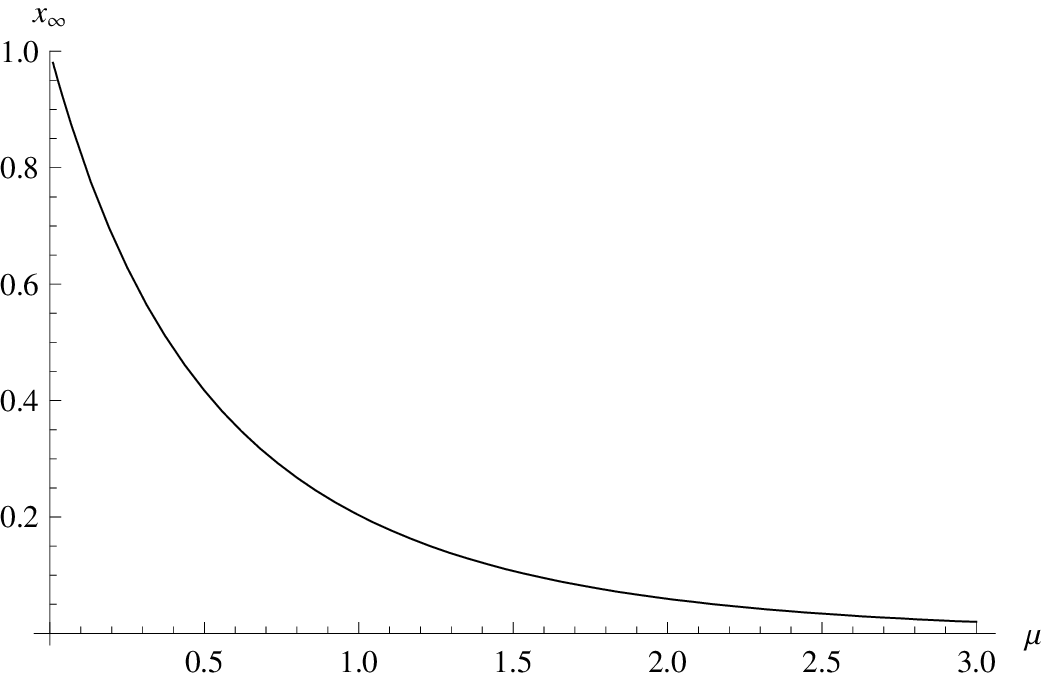}
\caption{Graph of $\xf(\mu, 1, 0)$.}
\label{Fig: Proportion}
\end{figure}

Here are some important cases:

\smallskip
\textbf{(a)} For $R \equiv \kappa$ ($\kappa \geq 1$ an integer), we have the $\kappa$-fold stifling Maki-Thompson model (so called by~\citet{dg}, in the context of the Daley-Kendall model), for which
\[ \sigma^2 = \frac{\xf (1 - \xf)}{1 - (\kappa + 1) \, \xf}, \]
where
\begin{equation}
\label{F: xf MT k}
\xf = \xf(\kappa, 1, 0) = -(\kappa + 1)^{-1} \, W_0(-(\kappa + 1) \, e^{-(\kappa + 1)}).
\end{equation}
Table~\ref{Tab: 1} exhibits the values of $\xf$ and $\sigma^2$ in this case for $\kappa = 1, \dots, 8$.
The original Maki-Thompson model is obtained by considering $R\equiv 1$, $x_0=1$ and $w_0=0$, consequently our theorems generalize classical results proved by \citet{sudbury} and \citet{watson}.
For the $2$-fold stifling Maki-Thompson model, the asymptotic value of $\xf \approx 0.0595$ was originally obtained by~\citet{carnal}.
Formula~\eqref{F: xf MT k} is presented in Appendix~D of Belen's doctorate thesis~\citep{belen08}.

\smallskip
\textbf{(b)} Let $R \sim \text{Geometric}(p)$, that is, $r_0 = 0$ and $r_i = p \, (1 - p)^{i - 1}$, $i = 1, 2, \dots$
In this model, an ignorant always becomes a spreader upon hearing the rumour and each time a spreader meets another spreader or a stifler, he decides with probability~$p$ to become a stifler, independently for each spreader and each meeting.
Thus, given that a spreader has not yet stopped propagating the rumour, the conditional distribution of the additional number of stifling experiences he will have does not depend on how many stifling experiences he already had.
This means that every time a spreader chooses whether or not to become a stifler, he does not have a ``memory'' of how many unsuccessful telling meetings he has been involved in.
Table~\ref{Tab: 2} shows the values of $\xf$ and~$\sigma^2$ for $x_0=1$, $w_0=0$ and some arbitrarily chosen values of $p$.

\smallskip
\textbf{(c)} Consider $R \sim \text{Poisson}(\lambda)$, in which case an ignorant individual has the choice (with a positive probability equal to $e^{-\lambda}$) of becoming a stifler as soon as he learns the rumour.
Moreover, in his successive decisions about stifling, a spreader does have some ``memory'' of the number of his previous stifling experiences.
Table~\ref{Tab: 3} presents the values of $\xf$ and~$\sigma^2$ for $x_0=1$, $w_0=0$ and some values of $\lambda$.
\end{obs}

\begin{table}[htb]
\begin{center}
\begin{small}
\begin{tabular}{|c|r@{.}lr@{.}lr@{.}lr@{.}lr@{.}lr@{.}lr@{.}lr@{.}l|}
\hline
$ \kappa $ & \multicolumn{2}{c}{1} & \multicolumn{2}{c}{2} & \multicolumn{2}{c}{3} & \multicolumn{2}{c}{4} & \multicolumn{2}{c}{5} & \multicolumn{2}{c}{6} & \multicolumn{2}{c}{7} & \multicolumn{2}{c|}{8} \\\hline
$ \xf $ & 0&203 & 0&0595 & 0&0198 & 0&00698 & 0&00252 & 0&000918 & 0&000336 & 0&000124 \\
$ \sigma^2 $ & 0&273 & 0&0681 & 0&0211 & 0&00718 & 0&00255 & 0&000923 & 0&000337 & 0&000124 \\
\hline
\end{tabular}
\end{small}
\caption{$\kappa$-fold stifling model, $x_0=1$ and $w_0=0$.}
\label{Tab: 1}
\bigskip
\begin{small}
\begin{tabular}{|c|r@{.}lr@{.}lr@{.}lr@{.}lr@{.}lr@{.}lr@{.}lr@{.}lr@{.}l|}
\hline
$ p $ & \multicolumn{2}{c}{0.1} & \multicolumn{2}{c}{0.2} & \multicolumn{2}{c}{0.3} & \multicolumn{2}{c}{0.4} & \multicolumn{2}{c}{0.5} & \multicolumn{2}{c}{0.6} & \multicolumn{2}{c}{0.7} & \multicolumn{2}{c}{0.8} & \multicolumn{2}{c|}{0.9} \\\hline
$ \xf $ & 0&0000167 & 0&00252 & 0&0139 & 0&0340 & 0&0595 & 0&0878 & 0&117 & 0&147 & 0&175 \\
$ \sigma^2 $ & 0&0000167 & 0&00268 & 0&0163 & 0&0427 & 0&0780 & 0&118 & 0&159 & 0&199 & 0&238 \\
\hline
\end{tabular}
\end{small}
\caption{$R \sim \text{Geometric}(p)$, $x_0=1$ and $w_0=0$.}
\label{Tab: 2}
\bigskip
\begin{small}
\begin{tabular}{|c|r@{.}lr@{.}lr@{.}lr@{.}lr@{.}lr@{.}lr@{.}lr@{.}lr@{.}lr@{.}l|}
\hline
$ \lambda $ & \multicolumn{2}{c}{0.1} & \multicolumn{2}{c}{0.3} & \multicolumn{2}{c}{0.5} & \multicolumn{2}{c}{0.7} & \multicolumn{2}{c}{0.9} & \multicolumn{2}{c}{1.1} & \multicolumn{2}{c}{1.3} & \multicolumn{2}{c}{1.5} & \multicolumn{2}{c}{1.7} & \multicolumn{2}{c|}{1.9} \\\hline
$ \xf $ & 0&824 & 0&577 & 0&417 & 0&309 & 0&233 & 0&178 & 0&138 & 0&107 & 0&0844 & 0&0668 \\
$ \sigma^2 $ & 2&908 & 1&654 & 1&012 & 0&654 & 0&440 & 0&307 & 0&219 & 0&160 & 0&119 & 0&0895 \\
\hline
\end{tabular}
\end{small}
\caption{$R \sim \text{Poisson}(\lambda)$, $x_0=1$ and $w_0=0$.}
\label{Tab: 3}
\end{center}
\end{table}

\section{Proofs}
\label{S: Proofs}

Here are the main ideas in the proofs of Theorems~\ref{T: LLN} and~\ref{T: CLT}. 
First, by means of a suitable time change of the process, we define a new process $\{\n{\tilde{V}}{N}(t)\}_{t\geq 0}$ with the same transitions as $\{\n{V}{N}(t)\}_{t\geq 0}$, so that they end at the same point of the state space. 
Next, we work with a reduced Markov chain obtained from $\{\n{\tilde{V}}{N}(t)\}_{t\geq 0}$ in order to apply the theory of density dependent Markov chains presented in \citet{ethier}. 
As the arguments follow a path similar to that presented in~\citet{kurtz}, we present only a brief sketch of the proofs. 

\subsection*{Time-changed process.} 
Since the distribution of $\n{X}{N}(\tau^{(N)})$ depends on the process $\{{V}^{(N)}(t)\}_{t\geq 0}$ only through the embedded Markov chain, we consider a time-changed version of the process. Let $\{\n{\tilde{V}}{N}(t)\}_{t\geq 0}$ be the infinite-dimensional continuous time Markov chain
$$\{(\n{\tilde{X}}{N}(t),\n{\tilde{Y}}{N}_1(t),\n{\tilde{Y}}{N}_2(t),\dots)\}_{t \geq 0}$$
with increments and corresponding rates given by
{\allowdisplaybreaks
\begin{align*}
&\text{increment} & &\text{rate} & &  \\
& (-1,0,0,\dots) & &r_0 \tilde X\\
& (-1, 0, \dots, \overset{i-1}{0},\overset{i}{1},\overset{i+1}{0},\dots) & & r_i \tilde X  & & i=1,2,\dots\\
& (0,0,\dots, \overset{i-1}{1},\overset{i}{-1}, \overset{i+1}{0}, \dots) & &(N-\tilde X) \, \tilde Y_i \, (\tilde Y)^{-1} & & i=2,3,\dots \\[0.1cm]
& (0,-1,0,\dots) & &(N-\tilde X) \, \tilde Y_1 \, (\tilde Y)^{-1}.& & 
\end{align*}}%
Furthermore, $\{\n{\tilde{V}}{N}(t)\}_{t\geq 0}$ can be defined in such a way that it has the same initial state and the same transitions as~$\{\n{V}{N}(t)\}_{t\geq 0}$, so both have the same embedded Markov chain. Thus, by defining
$$ \tilde{\tau}^{(N)}= \inf \{t: \n{\tilde{Y}}{N}(t) = 0 \},$$
we have that $\n{X}{N}(\tau^{(N)})=\n{\tilde{X}}{N}(\tilde{\tau}^{(N)}).$

\subsection*{Dimension reduction and deterministic limit.} 
In order to prove the desired limit theorems using Theorem 11.2.1 of \citet{ethier}, we work with a reduced Markov chain. 
We define 
$$\n{\tilde{W}}{N}(t)=\sum_{i=1}^{\infty} i \, \n{\tilde{Y}_i}{N}(t),$$
and note that the process $\{(\n{\tilde{X}}{N}(t),\n{\tilde{W}}{N}(t))\}_{t\geq 0}$ is a continuous time Markov chain with increments and rates given by
\begin{equation}
\label{F: RRP}
\begin{aligned}
&\text{increment}  & \quad &\text{rate}  	& \, & \\
&\ell_i=(-1,i) 	   & 	   &r_i \tilde X    & 	 &i = 0, 1, \dots \\
&\ell_{-1}=(0,-1)  & 	   &N - \tilde X.   & 	 &
\end{aligned}
\end{equation}

Now we define, for $t\geq 0$,
$$\n{\tilde{v}}{N}(t) = (\n{\tilde{x}}{N}(t),\n{\tilde{w}}{N}(t)) = N^{-1} (\n{\tilde{X}}{N}(t),\n{\tilde{W}}{N}(t)),$$
and consider 
$$\beta_{\ell_{-1}}({x},{w})=1- {x} \quad \text{and} \quad \beta_{\ell_i}({x},{w})=r_i {x}, \quad i=0,1,\dots$$
Notice that the rates in \eqref{F: RRP} can be written as
$$N \, \beta_{\ell_i} \left( \dfrac{\tilde X}{N}, \dfrac{\tilde W}{N} \right),$$
so $ \{ \tilde v^{(N)}(t) \}_{t \geq 0} $ is a density dependent Markov chain with possible transitions in the set $\{ \ell_{-1}, \ell_0, \ell_1, \dots \}$.

Now we use Theorem~11.2.1 of~\citet{ethier} to conclude that the process $\{\n{\tilde{v}}{N}(t)\}_{t\geq 0}$ converges almost surely as $ N \to \infty $ to a deterministic limit.
The drift function defined in~\citet{ethier} by $F(x,w)=\sum_{i=-1}^{\infty} \ell_i \, \beta_{\ell_i}(x,w)$ is in this case given by
$$F(x,w)=(-x,(\mu +1)x-1).$$
Hence the limiting deterministic system is governed by the following system of ordinary differential equations
\begin{equation*}
\begin{cases}
x^{\prime}(t) = -x(t), \\
w^{\prime}(t) = (\mu +1)x-1
\end{cases}
\end{equation*}
with initial conditions $x(0)=x_0$ and $w(0)=w_0$. The solution of this system is given by $v(t)=(x(t),w(t))$, where
$$x(t)=x_0 \, e^{-t} \, \text{ and } \, w(t)=f(x(t))=w_0+(1+\mu)\left(x_0 - x(t)\right)-t.$$

According to Theorem $11.2.1$ of \citet{ethier}, we have that on a suitable probability space,
$$\lim_{N \to \infty} \tilde v^{(N)}(t) = v(t) \quad \text{a.s.}$$ 
uniformly on bounded time intervals. In particular, it can be proved that
\begin{equation}
\label{F: Conv x}
\lim_{N \to \infty}\tilde x^{(N)}(t)= x(t) \quad \text{a.s.} 
\end{equation}
uniformly on~$ \bbR $.
See Lemma~3.6 in~\citet{kurtz} for an analogous detailed proof.

\subsection*{Proofs of Theorems~\ref{T: LLN} and~\ref{T: CLT}.} 
To prove both theorems, we use Theorem 11.4.1 of~\citet{ethier}. 
We adopt their notations, except for the Gaussian process $V$ defined on p.~458, that we would rather denote by $U=(U_x,U_w)$. 
Here, $\varphi(x, w) = w$, and
$$\tf = \inf \{t: w(t) \leq 0 \} = w_0 + (1 + \mu)(x_0 - \xf). $$
Moreover, 
\begin{equation}
\label{F: Der Neg}
\nabla \varphi(v(\tf)) \cdot F(v(\tf)) = w^{\prime} (\tf) = (\mu +1) \xf - 1 < 0.
\end{equation}

\noindent
\textit{Proof of Theorem \ref{T: LLN}.}
We note that $w_0 > 0$ and \eqref{F: Der Neg} imply that $w(\tf - \eps) > 0$ and $w(\tf + \eps) < 0$ for $0 < \eps < \tf$. Then, the almost sure convergence of~$\tilde{w}^{(N)}$ to~$w$ uniformly on bounded intervals yields that
\begin{equation}
\label{F: Conv tf}
\lim_{N \to \infty} \, \tilde \tau^{(N)} = \tf \quad \text{a.s.}
\end{equation}
In the case where $ w_0 = 0 $ and $ x_0 > (1+\mu)^{-1} $, this result is also valid because $ w^{\prime}(0) > 0$ and \eqref{F: Der Neg} still holds.
On the other hand, if $ w_0 = 0 $ and $ x_0 \leq (1+\mu)^{-1}$, then $w(t) < 0$ for all $t > 0$, and again the almost sure convergence of~$\tilde{w}^{(N)}$ to~$w$ uniformly on bounded intervals yields that $ \lim_{N \to \infty} \tilde \tau^{(N)} = 0 = \tf $ almost surely.
Therefore, as $X^{(N)}(\tau^{(N)}) = \tilde{X}^{(N)}(\tilde{\tau}^{(N)})$, we obtain Theorem~\ref{T: LLN} from~\eqref{F: Conv x} and~\eqref{F: Conv tf}.

\medskip
\noindent
\textit{Proof of Theorem \ref{T: CLT}.}
From Theorem~11.4.1 of~\citet{ethier}, we have that if $w_0>0$ or $w_0=0$ and $x_0>(1+\mu)^{-1}$, then
$\sqrt{N} \, (\tilde x^{(N)}(\tilde \tau^{(N)}) - \xf)$
converges in distribution as $N \to \infty$ to
\begin{equation}
\label{F: LD}
U_x(\tf) + \frac{\xf}{(\mu+1)\xf-1} \, U_w(\tf).
\end{equation}
The resulting normal distribution has mean zero, so, to complete the proof of Theorem \ref{T: CLT}, we need to calculate the corresponding variance.

To this end, we have to compute the covariance matrix $\Cov(U(\tf), U(\tf))$, a task that can be accomplished using a mathematical software.
The first step is to calculate the matrix of partial derivatives of the drift function $F$ and the matrix~$G$. 
We obtain
$$
\partial F(x, w) = 
\begin{pmatrix} 
-1 & 0 \\ 
(\mu + 1) & 0 
\end{pmatrix}$$
and 
$$G(x, w) = 
\begin{pmatrix}
x & -\mu x \\
-\mu x & (\nu^2 + \mu^2 -1)x +1
\end{pmatrix}.
$$
Next, we compute the solution $\Phi$ of the matrix equation
$$ \frac{\partial}{\partial t} \, \Phi(t, s) = \partial F(x(t), w(t)) \, \Phi(t, s),
\quad \Phi(s, s) = I_2, $$
which is given by
$$ 
\Phi(t, s) = 
\begin{pmatrix} 
e^{-(t-s)} & 0 \\ 
(\mu +1)(1- e^{-(t-s)}) & 1
\end{pmatrix}. 
$$
Hence, the covariance matrix of the Gaussian process $U$ at time~$t$ is obtained by the formula
\begin{equation}
\label{F: Covariance}
\Cov(U(t), U(t)) = \int_0^{t} \Phi(t, s) \, G(x(s), w(s)) \, {[\Phi(t, s)]}^T \, ds.
\end{equation}
As the final step to compute $\Cov(U(\tf), U(\tf))$, we have to replace $e^{-t}$ and $t$ in the formula obtained from~\eqref{F: Covariance} by $\xf / x_0$ and $\tf$, respectively.
The resulting formulas are 
{\allowdisplaybreaks
\begin{align*}
\V(U_x(\tf)) &= ((x_0 - \xf) \xf) / x_0, \\
\V(U_w(\tf)) &= {(\mu + 1)^2 (x_0 - \xf) \xf}/{x_0} + \nu^2 (x_0 - \xf) \\
& \quad + (1 - 2 (\mu + 1) \xf) \tf, \\
\Cov(U_x(\tf), U_w(\tf)) &= \tf \xf - {(\mu + 1) (x_0 - \xf) \xf}/{x_0}.
\end{align*}}%
Using that $\tf = w_0 + (1+\mu)(x_0-\xf)$, \eqref{F: LD} and well-known properties of the variance, we get formula~\eqref{F: Variance}.

\section{Concluding remarks}
\label{S: Concluding remarks}

We have proposed a general Maki-Thompson model in which an ignorant individual is allowed to have 
a random number of stifling experiences once he is told the rumour. 
The assigned numbers of stifling experiences are independent and identically distributed random variables with mean~$\mu$ and variance~$\nu^2$.
We prove that the ultimate proportion of ignorants converges in probability to an asymptotic value as the population size tends to~$\infty$.
A Central Limit Theorem describing the magnitude of the random fluctuations around this limiting value is also derived.
The asymptotic value and the variance of the Gaussian distribution in the CLT are functions of $\mu$, $\nu^2$ and some constants related to the initial state of the process.

We observe that in fact it is possible to obtain another result, concerning the mean number $\n{m}{N}$ of transitions that the process makes until absorption. 
Using an argument analogous to that presented in Theorem~2.5 of~\citet{kurtz}, it can be proved that, if $\nu^2<\infty$, then 
\[ \lim_{N\to \infty} {N}^{-1} \, {\n{m}{N}} = \tf = w_0 + (1+\mu)(x_0-\xf). \]

As a final remark, we would like to point out the usefulness of the theory of density dependent Markov chains as a tool for studying the limiting behaviour of stochastic rumour processes.
This approach constitutes an alternative to the pgf method and the Laplace transform presented in~\citet{dg}, \citet{gani} and \citet{pearce}.

\section*{Acknowledgments}

The authors are grateful to Tom Kurtz, Alexandre Leichsenring, Nancy Lopes Garcia, Pablo Groisman and Sebastian Grynberg for fruitful discussions.
Thanks are also due to three reviewers for their helpful comments.


\begin{thebibliography}{99}

\bibitem[Belen(2008)]{belen08}
Belen, S., 2008. 
The behaviour of stochastic rumours. 
Ph.D. thesis, School of Mathematical Sciences, University of Adelaide, Australia.
Available at \urlstyle{rm}\url{http://hdl.handle.net/2440/49472}.

\bibitem[Belen and Pearce(2004)]{belen04}
Belen, S., Pearce, C.E.M., 2004. 
Rumours with random initial conditions. 
\textit{The Australian \& New Zealand Industrial and Applied Mathematics Journal} \textbf{45}, 393--400.

\bibitem[Carnal(1994)]{carnal}
Carnal, H., 1994.
Calcul des probabilit\'es et mod\'elisation. 
\textit{Elemente der Mathematik} \textbf{49} (4), 166--173.

\bibitem[Corless et al.(1996)]{lambert}
Corless, R.M., Gonnet, G.H., Hare, D.E.G., Jeffrey, D.J., Knuth, D.E., 1996.
On the Lambert $W$ function. 
\textit{Advances in Computational Mathematics} \textbf{5} (4), 329--359.

\bibitem[Daley and Gani(1999)]{dg}
Daley, D.J., Gani, J., 1999. 
Epidemic Modelling: an Introduction.
Cambridge University Press, Cambridge. 

\bibitem[Daley and Kendall(1965)]{kendall}
Daley, D.J., Kendall, D.G., 1965.
Stochastic rumours. 
\textit{Journal of the Institute of Mathematics and its Applications} \textbf{1}, 42--55.

\bibitem[Dunstan(1982)]{dunstan}
Dunstan, R., 1982.
The rumour process. 
\textit{Journal of Applied Probability} \textbf{19} (4), 759--766.

\bibitem[Ethier and Kurtz(1986)]{ethier}
Ethier, S.N., Kurtz, T.G., 1986.
Markov Process: Characterization and Convergence. Wiley Series in Probability and 
Mathematical Statistics. John Wiley \& Sons Inc., New York.

\bibitem[Gani(2000)]{gani}
Gani, J., 2000.
The Maki-Thompson rumour model: a detailed analysis. 
\textit{Environmental Modelling \& Software} \textbf{15}, 721--725.

\bibitem[Kurtz et al.(2008)]{kurtz}
Kurtz, T.G., Lebensztayn, E., Leichsenring, A.R., Machado, F.P., 2008.
Limit theorems for an epidemic model on the complete graph. 
\textit{ALEA. Latin American Journal of Probability and Mathematical Statistics} \textbf{4}, 45--55.

\bibitem[Lebensztayn et al.(2010)]{EFP}
Lebensztayn, E., Machado, F.P., Rodr\'iguez. P.M., 2010.
Limit theorems for a general stochastic rumour model. 
Available at \urlstyle{rm}\url{http://arxiv.org/abs/1003.4995}.

\bibitem[Lefevre and Picard(1994)]{picard}
Lefevre, C., Picard, P., 1994.
Distribution of the final extent of a rumour process. 
\textit{Journal of Applied Probability} \textbf{31} (1), 244--249.

\bibitem[Maki and Thompson(1973)]{maki}
Maki, D.P., Thompson, M., 1973.
Mathematical Models and Applications. Prentice-Hall, Englewood Cliffs.

\bibitem[Pearce(2000)]{pearce}
Pearce, C.E.M., 2000.
The exact solution of the general stochastic rumour. 
\textit{Mathematical and Computer Modelling} \textbf{31}, 289--298.

\bibitem[Sudbury(1985)]{sudbury}
Sudbury, A., 1985.
The proportion of the population never hearing a rumour. 
\textit{Journal of Applied Probability} \textbf{22} (2), 443--446.

\bibitem[Watson(1988)]{watson}
Watson, R., 1988.
On the size of a rumour. 
\textit{Stochastic Processes and their Applications} \textbf{27}, 141--149.

\end{thebibliography}
\end{document}